\documentclass[12pt]{amsart}

\usepackage{amssymb,latexsym}

\usepackage{enumerate}
\usepackage{amsmath}

\makeatletter

\@namedef{subjclassname@2010}{

  \textup{2010} Mathematics Subject Classification}

\usepackage{geometry}
\newtheorem{proposition}{Proposition}[section]
\newtheorem{theorem}{Theorem}[section]

\newtheorem{lemma}{Lemma}[section]

\newtheorem{thm}{Theorem}[section]
\newcommand{\re}{\textup{Re}}
\newcommand{\im}{\textup{Im}}

\newcommand{\tmop}[1]{\ensuremath{\operatorname{#1}}}
\newcommand{\ex}{\mathbb{E}}
\newtheorem*{theoremvor}{Voronin's universality theorem}
\numberwithin{equation}{section}

\begin{document}

\baselineskip=17pt

\title{An effective universality theorem for the Riemann zeta function}

\author{Youness Lamzouri}
\author{Stephen Lester}
\author{Maksym Radziwill}

\address{Department of Mathematics and Statistics,
York University,
4700 Keele Street,
Toronto, ON,
M3J1P3
Canada}

\email{lamzouri@mathstat.yorku.ca}

\address{Department of Mathematics and Statistics,  University of Montreal
Pavillon Andr\'e-Aisenstadt,
PO Box 6128, Centre-ville Station Montreal, 
Quebec H3C 3J7}

\email{sjlester@gmail.com}
\address{Department of Mathematics
McGill University
805 Sherbrooke Street West
Montreal, Quebec H3A 0G4}
\email{maksym.radziwill@gmail.com}

\begin{abstract}
Let $0<r<1/4$, and $f$ be a non-vanishing continuous function in  $|z|\leq r$, that is analytic in the interior. Voronin's
universality theorem asserts that translates of the Riemann zeta function $\zeta(3/4 + z + it)$ can approximate
$f$ uniformly in $|z| < r$ to any given precision $\varepsilon$, and moreover that the set of such $t \in [0, T]$ has measure at least $c(\varepsilon) T$ for some 
 $c(\varepsilon) > 0$, once $T$ is large enough. This was refined by Bagchi who showed that the measure of such $t \in [0,T]$ is $(c(\varepsilon) + o(1)) T$, for all but at most countably many $\varepsilon > 0$. 
Using a completely different approach, we obtain the first effective version of Voronin's Theorem, by showing that in the rate of convergence one can save a small power of the logarithm of $T$. Our method is flexible, and can be generalized to other $L$-functions in the $t$-aspect, as well as to families of $L$-functions in the conductor aspect.

\end{abstract}


\thanks{The first and third authors are partially supported by  Discovery Grants from the Natural Sciences and Engineering Research Council of Canada.}

\maketitle

\section{Introduction}

In 1914 Fekete constructed a formal power series $\sum_{n = 1}^{\infty} a_n x^n$ with the following \textit{universal} property: For any continuous function $f$ on $[-1,1]$ (with $f(0) = 0$) and given any $\varepsilon > 0$ there exists an integer $N > 0$ such that
$$
\sup_{-1 \leq x \leq 1} \Big | \sum_{n \leq N} a_n x^n - f(x) \Big | < \varepsilon.
$$
In the 1970's Voronin \cite{Voronin} discovered the remarkable fact that the Riemann zeta-function satisfies a similar universal property. He showed that for any $r < \tfrac 14$, any non-vanishing continuous function $f$ in $|z| \leq r$, which is analytic in the interior, and for arbitrary $\varepsilon>0$, there exists a $T > 0$ 
such that
\begin{equation} \label{universal}
\max_{|z| \leq r} \Big | \zeta(\tfrac 34 + i T + z) - f(z) \Big | < \varepsilon.
\end{equation}
Voronin obtained a more quantitative description of this phenomena, stated below.\begin{theoremvor}
Let $0 < r < \tfrac 14$ be a real number. Let $f$ be a non-vanishing continuous function in $|z|\leq r$, that is analytic in the interior. Then, for any $\varepsilon > 0$, 
\begin{equation}\label{LIMINF}
\liminf_{T \rightarrow \infty}
\frac{1}{T} \cdot  \textup{meas} \Big \{ T \leq t \leq 2T: \max_{|z| \leq r}
\Big | \zeta(\tfrac 34 + it + z) - f(z) \Big | < \varepsilon \Big \}
> 0,
\end{equation}
where $\textup{meas}$ is Lebesgue's measure on $\mathbb{R}$.
\end{theoremvor}
There are several extensions of this theorem, for example to domains more general than compact discs (such as any compact set $K$ contained in the strip $1/2 < \re(s) < 1$ and with connected complement), or to more general $L$-functions. For a complete history of this subject, we refer the reader to \cite{Matsumoto}.

The assumption that $f(z) \neq 0$ is necessary: 
if $f$ were allowed to vanish then an application of Rouche's theorem would produce at least $\asymp T$ zeros $\rho = \beta + i \gamma$ of $\zeta(s)$ with $\beta > \tfrac 12 + \varepsilon$ and $T \leq \gamma \leq 2T$, contradicting the simplest zero-density theorems. 

Subsequent work of Bagchi \cite{Bagchi} clarified Voronin's universality theorem by setting it in the context of probability theory (see \cite{Kowalski} for a streamlined proof). Viewing $\zeta(\tfrac 34 + it + z)$ with $t \in [T, 2T]$ as a random variable $X_T$ in the space of random analytic functions (i.e $X_T(z) = \zeta(\tfrac 34 + i U_T + z)$ with $U_T$ uniformly distributed in $[T, 2T]$), Bagchi showed that as $T \rightarrow \infty$ this sequence of random variables converges in law (in the space of random analytic functions) to a random Euler product, 
$$
\zeta(s, X) := \prod_{p} \Big (1 - \frac{X(p)}{p^s} \Big )^{-1}
$$
with $\{X(p)\}_p$ a sequence of independent random variables uniformly distributed on the unit circle (and with $p$ running over prime numbers). This product converges almost surely for $\re(s) > \tfrac 12$ and defines almost surely a holomorphic function in the half-plane $\re(s) > \sigma_0$ for any $\sigma_0 > \tfrac 12$ (see Section 2 below). The proof of Voronin's universality is then reduced to showing that the support of $\zeta(s+3/4, X)$ in the space of random analytic functions contains all non-vanishing analytic $f : \{ z : |z| < r\} \rightarrow \mathbb{C} \backslash \{0\}$. Moreover it follows from Bagchi's work that the limit in Voronin's universality theorem exists for all but at most countably many $\varepsilon > 0$. 

In this paper, we present an alternative approach to Bagchi's result using methods from hard analysis. As a result we obtain, for the first time, a rate of convergence in Voronin's universality theorem. We also give an explicit description for the limit in terms of the random model $\zeta(s, X)$. 

\begin{theorem} \label{thm:main}
Let $0 < r < \tfrac 14$. Let $f$ be a non-vanishing continuous function on $|z|\leq (r+1/4)/2$ that is  holomorphic in $|z| < (r+1/4)/2$. Let $\omega$ be a real-valued continuously differentiable function with compact support. Then, we have
\begin{align*}
\frac{1}{T}\int_{T}^{2T}\omega\left(\max_{|z| \leq r} |\zeta(\tfrac 34 + it + z) - f(z)|  \right)dt 
& = \ex\left(\omega\left(\max_{|z|\leq r}|\zeta(\tfrac 34 + z, X)-
f(z)|\right)\right)\\
&+ O\left((\log T)^{-\frac{(3/4-r)}{11}+o(1)} \right),
\end{align*}
where the constant in the $O$ depends on $f, \omega$ and $r$.
\end{theorem}
If the random variable $Y_{r, f}=\max_{|z|\leq r}|\zeta(\tfrac 34 + z, X)-f(z)|$ is absolutely continuous, then it follows from the proof of Theorem \ref{thm:main} that for any fixed $\varepsilon>0$ we have
\begin{align*}
&\frac{1}{T} \cdot  \textup{meas} \Big \{ T \leq t \leq 2T: \max_{|z| \leq r}
\Big | \zeta(\tfrac 34 + it + z) - f(z) \Big | < \varepsilon \Big \}\\
 &= \mathbb{P}\left(\max_{|z|\leq r}\big|\zeta(\tfrac 34 + z, X)-
f(z)\big|<\varepsilon\right)
+ O\left((\log T)^{-\frac{(3/4-r)}{11}+o(1)} \right).
\end{align*}
Unfortunately, we have not been able to even show that $Y_{r, f}$ has no jump discontinuities. We conjecture the latter to be true, and one might even hope that $Y_{r, f}$ is absolutely continuous. 

A slight modification of the proof of Theorem \ref{thm:main}   allows for more general domains than the disc $|z| \leq  r$. 
Furthermore, if $\omega \geq \mathbf{1}_{(0, \varepsilon)}$ (where $\mathbf{1}_{S}$ is the indicator function of the set $S$), then it follows from Voronin's universality theorem that the main term in Theorem \ref{thm:main} is positive. Explicit lower bounds for the limit in \eqref{LIMINF} (in terms of $\varepsilon$) are contained in the papers of Good \cite{Good} and Garunkstis \cite{Garunkstis}.  

Our approach is flexible, and can be generalized to other $L$-functions in the $t$-aspect, as well as to ``natural'' families of $L$-functions in the conductor aspect. The only analytic ingredients that are needed are zero density estimates, and bounds on the coefficients of these $L$-functions (the so-called Ramanujan conjecture).  In particular, the techniques of this paper can be used to obtain an effective version of a recent result of Kowalski \cite{Kowalski}, who proved an analogue of Voronin's universality theorem for families of $L$-functions attached to $GL_2$ automorphic forms. In fact, using the zero-density estimates near $1$ that are known for a very large class of $L$-functions (including those in the Selberg class by Kaczorowski and Perelli \cite{KP}, and for families of $L$-functions attached to $GL_n$ automorphic forms by Kowalski and Michel \cite{KM}), one can prove an analogue of Theorem \ref{thm:main} for these $L$-functions, where we replace $3/4$ by some $\sigma<1$ (and $r<1-\sigma$). 


The main idea in the proof of Theorem \ref{thm:main} is to cover the boundary of the disc  $|z|\leq r$ with a union of a growing (with $T$) number of discs, while maintaining a global control of the size of $|\zeta'(s + z)|$ on $|z| \leq r$. It is enough to focus on the boundary of the disc thanks to the maximum modulus principle.
The behavior of $\zeta(s + z)$ with $z$ localized to a shrinking disc is essentially governed by the behavior at a single point $z = z_i$ in the disc. 
This allows us to reduce the problem to understanding the joint distribution of a growing number of shifts $\log\zeta(s + z_i)$ with the $z_i$ well-spaced, which can be understood by computing the moments of these shifts and using standard Fourier techniques. 

It seems very difficult to obtain a rate of convergence which is better than logarithmic in Theorem \ref{thm:main}. We have at present no understanding as to what the correct rate of convergence should be. 

\section{Key Ingredients and detailed results} \label{sec:propositions}

We first begin with stating certain important properties of the random model $\zeta(s, X)$. Let $\{X(p)\}_p$ be a sequence of independent random variables uniformly distributed on the unit circle. Then we have
$$-\log\left(1-\frac{X(p)}{p^s}\right)=\frac{X(p)}{p^s}+ h_X(p,s),$$
where the random series
\begin{equation}\label{ErrorRandom}
\sum_{p} h_X(p,s),
\end{equation}
converges absolutely and surely for $\re(s)>1/2$. Hence, it (almost surely) defines a holomorphic function in $s$ in this half-plane. Moreover, since $\ex(X(p))=0$ and $\ex(|X(p)|^2)=1$, then it follows from Kolmogorov's three-series theorem that the series
\begin{equation}\label{MainRandom}
\sum_{p}\frac{X(p)}{p^s}
\end{equation}
is almost surely convergent for $\re(s)>1/2$. By well-known results on Dirichlet series, this shows that this series defines (almost surely) a holomorphic function on the half-plane $\re(s)>\sigma_0$, for any $\sigma_0>1/2$. Thus, by taking the exponential of the sum of the random series in \eqref{ErrorRandom} and \eqref{MainRandom}, it follows that $\zeta(s, X)$ converges almost surely to a holomorphic function on the half-plane $\re(s)>\sigma_0$, for any $\sigma_0>1/2$. 

We extend the $X(p)$ multiplicatively to all positive integers by setting $X(1)=1$ and 
$X(n):= X(p_1)^{a_1}\cdots X(p_k)^{a_k}, \text{ if }  n= p_1^{a_1}\dots p_k^{a_k}.$
Then we have 
\begin{equation}\label{orthogonality}
\ex\left(X(n)\overline{X(m)}\right)=\begin{cases} 1 & \text{if } m=n,\\
0 & \text{otherwise}.\\
\end{cases}
\end{equation}
Furthermore, for any complex number $s$ with $\re(s)>1/2$ we have almost surely that
$$ \zeta(s, X)= \sum_{n=1}^{\infty} \frac{X(n)}{n^s}.$$

To compare the distribution of $\zeta(s+it)$ to that of $\zeta(s, X)$, we define a probability measure on $[T, 2T]$ in a standard way, by  
$$\mathbb{P}_T(S):= \frac{1}{T}\textup{meas}(S), \text{ for any } S\subseteq [T, 2T].$$

The idea behind our proof of effective universality is to
first reduce the problem to the discrete problem of controlling the
distribution of many shifts $\log \zeta(s_j + it)$ with 
all of the $s_j$ contained in a compact set inside the strip $\tfrac 12 < \re(s) < 1$. 
One of the main ingredients in this reduction is
the following result which allows us
to control the maximum of the derivative of the Riemann zeta-function.
This is proven in Section \ref{sec:derivative}.

\begin{proposition}\label{ControlDerivative}
Let $0<r<1/4$ be fixed. Then there exist positive constants $b_1$, $b_2$ and $b_3$  (that depend only on $r$) such that 
$$
\mathbb{P}_T \left( \max_{|z|\leq r} |\zeta'(\tfrac 34 + it + z)| > e^{V}
\right) \ll \exp\bigg( -b_1 V^{\frac{1}{1-\sigma(r)}}(\log V)^{\frac{\sigma(r)}{1-\sigma(r)}}\bigg)
$$
where $\sigma(r)=\tfrac34-r$,
uniformly for $V$ in the range $b_2<V \leq b_3 (\log T)^{1-\sigma}/(\log \log T)$. 
\end{proposition}
We also prove an analogous result for the random model $\zeta(s, X)$, which holds for all sufficiently large $V$. 
\begin{proposition}\label{DerRandom}
Let $0<r<1/4$ be fixed and $\sigma(r)=\tfrac34-r$.  Then there exist positive constants $b_1$ and $b_2$ (that depend only on $r$) such that for all $V>b_2$ we have
$$\mathbb{P}\left(\max_{|z| \leq r} |\zeta'(\tfrac 34 + z, X)| > e^V\right) \ll  \exp\bigg( -b_1 V^{\frac{1}{1-\sigma(r)}}(\log V)^{\frac{\sigma(r)}{1-\sigma(r)}}\bigg).$$
\end{proposition}

Once the reduction has been accomplished, it remains to understand the joint
distribution of the shifts $\{\log \zeta(s_1 + it), \log \zeta(s_2 + it), \dots, \log \zeta(s_J + it)\}$
with $J \rightarrow \infty$ as $T \rightarrow \infty$ at a certain rate,
and $s_1, \ldots, s_J$ are complex numbers with $\tfrac 12 < \re (s_j) < 1$ for all $j \leq J$. Heuristically, this should be well approximated by the the joint distribution of the random variables
$\{\log \zeta(s_1, X),\log \zeta(s_2, X), \dots, \log \zeta(s_J, X)\}$. In order to establish this fact (in a certain range of $J$), we first prove, in Section \ref{sec:moments}, that the
moments of the joint shifts $\log\zeta(s_j+it)$ are very close to the corresponding ones of $\log\zeta(s_j, X)$, for $j\leq J$.
\begin{thm} \label{MomentsShifts} 
Fix $1/2<\sigma_0<1$. Let $s_1, s_2, \dots, s_k, r_1, \dots, r_{\ell}$ be complex numbers in the rectangle $\sigma_0<\re(z)<1$ and $|\im(z)|\leq T^{(\sigma_0-1/2)/4}$.  Then, there exist positive constants $c_3, c_4, c_5 $ and a set $\mathcal{E}(T)\subset [T,2T]$ of measure $\ll T^{1-c_3}$, such that if $k, \ell \leq c_4\log T/\log\log T$ then 
\begin{align*}
&\frac{1}{T} \int_{[T,2T]\setminus \mathcal{E}(T)}\left(\prod_{j=1}^k\log\zeta(s_j+it)\right)\left(\prod_{j=1}^{\ell}\log\zeta(r_j-it)\right)dt\\
&=  \ex\left(\left(\prod_{j=1}^k\log\zeta(s_j,X)\right)\left(\prod_{j=1}^{\ell}\log\overline{\zeta(r_j, X)}\right)\right)+ O\left(T^{-c_5}\right).
\end{align*}
\end{thm}
Having obtained the moments we are in position to understand the
characteristic function, 
$$ \Phi_T(\mathbf{u}, \mathbf{v}):= \frac1T \int_T^{2T} \exp\left(i\left(\sum_{j=1}^J (u_j \re\log\zeta(s_j+it)+ v_j \im\log\zeta(s_j+it))\right)\right)dt,$$
where $\mathbf{u}=(u_1,\dots, u_J)\in \mathbb{R}^J$ and $\mathbf{v}=(v_1,\dots, v_J)\in \mathbb{R}^J$. We relate the above characteristic function to
the characteristic function of the probabilistic model, 
$$ \Phi_{\textup{rand}}(\mathbf{u}, \mathbf{v}):= \ex\left( \exp\left(i\left(\sum_{j=1}^J (u_j \re\log\zeta(s_j, X)+ v_j \im\log\zeta(s_j, X))\right)\right)\right).$$
This is obtained in the following theorem, which we prove in Section \ref{sec:characteristic}. 

\begin{thm}\label{characteristic}
Fix $1/2<\sigma<1$. Let $T$ be large and $J\leq (\log T)^{\sigma}$ be a positive integer. Let $s_1, s_2, \dots, s_J$ be complex numbers such that  $\min(\re(s_j))=\sigma$ and $\max(|\im(s_j)|)<T^{(\sigma-1/2)/4}$. Then, there exist positive constants $c_1(\sigma), c_2(\sigma)$, such that for all $ \mathbf{u}, \mathbf{v} \in \mathbb{R}^J$ such that $\max(|u_j|), \max(|v_j|)\leq c_1(\sigma) (\log T)^{\sigma}/J$  we have 
$$ \Phi_T(\mathbf{u}, \mathbf{v})= \Phi_{\textup{rand}}(\mathbf{u}, \mathbf{v})+ O\left(
\exp\left(-c_2(\sigma)\frac{\log T}{\log\log T}\right)\right).$$
\end{thm}
Using this result, we can show that the joint distribution of the shifts $\log \zeta(s_j+it)$ is very close to the corresponding joint distribution of the random variables $\log \zeta(s_j, X)$. The proof depends on Beurling-Selberg functions.
To measure how close are these distributions, we introduce the discrepancy $\mathcal{D}_T(s_1, \ldots, s_J)$ defined as 
\begin{align*}
\sup_{(\mathcal R_1, \ldots, \mathcal R_J) \subset \mathbb C^J} \bigg| \mathbb P_T\bigg( \log \zeta(s_j+it) \in \mathcal R_j, \forall j \le J  \bigg)-
\mathbb P\bigg( \log \zeta(s_j, X) \in \mathcal R_j, \forall j \le J   \bigg)\bigg| 
\end{align*}
where the supremum is taken over all $(\mathcal R_1, \ldots, \mathcal R_J) \subset \mathbb C^J$ and 
for each $j=1,\ldots, J$ the set $\mathcal R_j$ 
is a rectangle with sides parallel to the coordinate axes.
Our next theorem, proven in Section \ref{sec:distribution}, states a bound
for the above discrepancy. This generalizes Theorem 1.1 of \cite{LLR}, which corresponds to the special case $J=1$. 

\begin{thm} \label{discrep}
Let $T$ be large, $\tfrac12<\sigma<1$ and $J \le (\log T)^{\sigma/2}$ be a positive integer.
Let $s_1,s_2, \ldots, s_J$ be complex numbers such that
$$
\tfrac12 < \sigma:=\min_j(\tmop{Re}(s_j)) \le \max_j(\tmop{Re}(s_j)) <1 
\quad \mbox{and} \quad \max_j(|\tmop{Im}(s_j)|) < T^{ (\sigma-\frac12)/4}.
$$
 Then, we have
\begin{align*}
\mathcal{D}_T(s_1, \ldots, s_J) \ll \frac{J^2 }{(\log T)^{\sigma}}.
\end{align*}
\end{thm}

With all of the above tools in place we are ready to prove Theorem \ref{thm:main}. This is accomplished in the next section. 

\section{Effective universality: Proof of Theorem \ref{thm:main}} \label{sec:proof}

In this section, we will prove Theorem \ref{thm:main} using the results 
described in Section \ref{sec:propositions}. First, by the maximum modulus principle, the maximum of $|\zeta(\tfrac 34 + it + z) -
f(z)|$ in the disc $\{z: |z|\leq r\}$ must occur on its boundary $\{z: |z|=r\}$.
Our idea consists of first covering the circle $|z| = r$ with $J$ discs of radius $\varepsilon$ and centres $z_j$, where $z_j\in \{z: |z|=r\} $ for all $1\leq j\leq J$, and $J \asymp 1/\varepsilon$.  We call each of the discs
$\mathcal{D}_j$. Then, we observe that  
\begin{equation}\label{ComparisonSupMax}
\max_{j\leq J}|\zeta(\tfrac 34 + it + z_j) -
f(z_j)|\leq \max_{|z| \leq r} |\zeta(\tfrac 34 + it + z) - f(z)|\leq \max_{j\leq J}\max_{z\in \mathcal{D}_j} |\zeta(\tfrac 34 + it + z) - f(z)|.
\end{equation}
Using Proposition \ref{ControlDerivative}, we shall prove that for all $j\leq J$ (where $J$ is a small power of $\log T$) we have
$$\max_{z\in \mathcal{D}_j} |\zeta(\tfrac 34 + it + z) - f(z)|\approx |\zeta(\tfrac 34 + it + z_j)-
f(z_j)|$$ for all $t\in [T, 2T]$ except for a set of points $t$ of very small measure. We will then deduce that the (weighted) distribution of $\max_{|z| \leq r} |\zeta(\tfrac 34 + it + z) - f(z)|$ is very close to the corresponding distribution of $\max_{j\leq J}|\zeta(\tfrac 34 + it + z_j)-
f(z_j)|$, for $t\in [T, 2T]$. We will also establish an analogous result for the random model $\zeta(s, X)$ along the same lines, by using Proposition \ref{DerRandom} instead of Proposition \ref{ControlDerivative}. Therefore, to complete the proof of Theorem \ref{thm:main} we need to compare the distributions of $\max_{j\leq J}|\zeta(\tfrac 34 + it + z_j)-f(z_j)|$ and $\max_{ j\leq J}|\zeta(\tfrac 34 + z_j, X)-
f(z_j)|$. Using Theorem \ref{discrep} we prove
\begin{proposition}\label{DistributionMax} 
Let $T$ be large, $0<r<1/4$ and $J\leq (\log T)^{(3/4-r)/7}$ be a positive integer. Let $z_1, \dots, z_J$ be complex numbers such that $|z_j|\leq r$. Then we have 
\begin{align*}
&\left|\mathbb{P}_T\left(\max_{j\leq J}|\zeta(\tfrac 34 + it + z_j)-
f(z_j)|\leq u\right)-\mathbb P\left(\max_{ j\leq J}|\zeta(\tfrac 34 + z_j, X)-
f(z_j)|\leq u\right)\right|\\
&\ll_u \frac{(J\log\log T)^{6/5}}{(\log T)^{(3/4-r)/5}}.
\end{align*}
\end{proposition}
\begin{proof}
Fix a positive real number $u$. Let $\mathcal{A}_J(T)$ be the set of those $t$ for which
 $|\arg\zeta(\tfrac 34 + it + z_j)|\leq \log\log T$ for every $j\leq J$. Since $\re(\tfrac 34 + it + z_j)\geq \tfrac34-r$ and $\im(\tfrac 34 + it + z_j)=t+O(1)$, then it follows from Theorem 1.1 and Remark 1 of \cite{La} that for each $j\leq J$ we have
\begin{equation}\label{LargeDeviationArg}
 \mathbb{P}_T\left(|\arg \zeta(\tfrac 34 + it + z_j)|\geq \log\log T\right)
 \ll \exp\left(-(\log\log T)^{(\tfrac14+r)^{-1}}\right)\ll \frac{1}{(\log T)^4}.
\end{equation}
Therefore, we obtain
 $$\mathbb{P}_T\left([T, 2T]\setminus\mathcal{A}_J(T)\right)\leq \sum_{j=1}^J 
 \mathbb{P}_T\left(|\arg \zeta(\tfrac 34 + it + z_j)|\geq \log\log T\right)\ll \frac{J}{(\log T)^4}\ll \frac{1}{(\log T)^2},$$
and this implies that
\begin{equation} \label{mainterm}
\begin{aligned}\mathbb{P}_T\left(\max_{j\leq J}|\zeta(\tfrac 34 + it + z_j)-
f(z_j)|\leq u\right)&=\mathbb{P}_T \left(\max_{ j\leq J} |\zeta(\tfrac 34 + it + z_j) - f(z_j)
| \leq  u \ , \ t \in \mathcal{A}_J(T)\right)\\
&+O\left(\frac{1}{(\log T)^2}\right).
\end{aligned} 
\end{equation}

For each $j\leq J$ consider the region
$$
\mathcal U_j=\left\{ z: 
|e^z - f(z_j)| \leq  u \ , \ |\im(z) | \leq  \log\log T \right\}.
$$
We
cover $\mathcal U_j$ with $K \asymp \tmop{area}(\mathcal U_j) / \varepsilon^2 \asymp \log \log T/\varepsilon^2$ squares
$\mathcal{R}_{j, k}$ with
sides of length $\varepsilon=\varepsilon(T)$, where $\varepsilon$ is a small positive parameter to be chosen later. Let $\mathcal K_j$ denote the set of $k \in \{1, 2, \ldots, K\}$ such that the intersection of
 $\mathcal R_{j,k}$ with the boundary of $\mathcal U_j$ is empty and write $\mathcal K_j^c$ for the relative complement of $\mathcal K_j$ with respect to $\{1, 2, \ldots, K\}$. Note that $ |\mathcal K_j^c| \asymp \log \log T/\varepsilon$. By construction,
\[
\left( \bigcup_{k \in  \mathcal K_j} \mathcal R_{j,k} \right)
\subset
\mathcal U_j \subset \left( \bigcup_{k \le K} \mathcal R_{j,k} \right).
\]
Therefore (\ref{mainterm}) can be expressed as
$$
\mathbb{P}_T \Big ( \forall j \leq J, \forall k \leq K:
\log \zeta(\tfrac 34 + it + z_j) \in \mathcal{R}_{j,k} \Big) +
\mathcal{E}_1 
$$
where by Theorem \ref{discrep}
\begin{equation}
\begin{split}
\mathcal{E}_1 \ll & \sum_{j \leq J} \sum_{k \in \mathcal K_j^c}
\mathbb{P}_{T} \Big ( \log \zeta(\tfrac 34 + it + z_j) \in \
\mathcal{R}_{j,k} \Big ) \\
\ll& \sum_{j \leq J} \sum_{k \in \mathcal K_j^c} \left(
\mathbb{P}_{T} \Big ( \log \zeta(\tfrac 34+z_j,X) \in \
\mathcal{R}_{j,k} \Big ) +\frac{1}{(\log T)^{3/4-r}} \right) \\
\ll& J \cdot \frac{ \log \log T}{\varepsilon} \left( \varepsilon^2+\frac{1}{ (\log T)^{3/4-r}} \right),
\end{split}
\end{equation}
and in the last step we used the fact that
 $\log \zeta(s, X)$ is an absolutely continuous random variable (see for example Jessen and Wintner \cite{JeWi}). 
We conclude that
\begin{equation}\label{CoveringRectanglesZeta}
\begin{aligned}
\mathbb{P}_T  \Big (\max_{j\leq J} |\zeta(\tfrac 34 + it + z_j) - f(z_j)
| \leq  u \Big) & =  \mathbb{P}_T \Big ( \forall j \leq  J, 
\forall k \leq  K: 
\log \zeta(\tfrac 34 + it + z_j) \in \mathcal{R}_{j,k} \Big ) \\
&   + O\left( \varepsilon J \log\log T + \frac{J\log\log T}{\varepsilon (\log T)^{3/4-r}} \right).
\end{aligned}
\end{equation}
Additionally, it follows from Theorem \ref{discrep} that  the main term of this last estimate
equals
\begin{equation} \label{eq:mainterm}
\mathbb{P} \Big ( \forall j \leq  J, \forall k \leq K:
\log \zeta(\tfrac 34 + z_j, X) \in \mathcal{R}_{j, k} \Big) +
O \left( \frac{J^2 (\log\log T)^2}{\varepsilon^4(\log T)^{3/4-r}} \right).
\end{equation}

We now repeat the exact same argument but for the random model $\zeta(s, X)$ instead of the zeta function. In particular, instead of \eqref{LargeDeviationArg} we shall use that 
$$
\mathbb{P} \left(|\arg \zeta(\tfrac 34 + z_j, X)|\geq \log\log T\right)
 \ll \exp\left(-(\log\log T)^{(\tfrac14+r)^{-1}}\right)\ll \frac{1}{(\log T)^4},
$$ 
 which follows from Theorem 1.9 of \cite{La}. Thus, similarly to \eqref{CoveringRectanglesZeta} we obtain
\begin{align*} 
\mathbb{P} \Big ( \forall j \leq  J, \forall k \leq K:
\log \zeta(\tfrac 34 + z_j, X) \in \mathcal{R}_{j, k} \Big) &= \mathbb{P}
\left(\max_{j\leq J} |\zeta(\tfrac 34 + z_j, X) - f(z_j)
| \leq  u\right)\\
&  + O\left( \varepsilon J \log\log T + \frac{J\log\log T}{\varepsilon (\log T)^{3/4-r}} \right).
\end{align*}
Combining  the above estimate with \eqref{CoveringRectanglesZeta} and \eqref{eq:mainterm} we conclude that
\begin{align*}
\mathbb{P}_T  \Big (\max_{j\leq J} |\zeta(\tfrac 34 + it + z_j) - f(z_j)
| \leq  u \Big) & =  \mathbb{P}
\left(\max_{j\leq J} |\zeta(\tfrac 34 + z_j, X) - f(z_j)
| \leq  u\right) \\
&  + O\left( \frac{J^2 (\log\log T)^2}{\varepsilon^4(\log T)^{3/4-r}}+ \varepsilon J \log\log T\right).
\end{align*}
Finally, choosing 
$$\varepsilon= \left(\frac{J\log\log T}{(\log T)^{3/4-r}}\right)^{1/5}$$
completes the proof.
\end{proof}

\begin{proof}[Proof of Theorem \ref{thm:main}]
We wish to estimate 
\begin{equation} \label{toestimate}
\frac{1}{T}\int_{T}^{2T}\omega\left(\max_{|z| \leq r} |\zeta(\tfrac 34 + it + z) - f(z)|  \right)dt
\end{equation}
with $f$ an analytic non-vanishing function, and where $\omega$ is a continuously differentiable function with compact support. 

Recall that the maximum of $|\zeta(\tfrac 34 + it + z) -
f(z)|$
on the disc $\{z: |z|\leq r\}$ must occur on its boundary $\{z: |z|=r\}$, by the maximum modulus principle. Let $\varepsilon \leq (1/4-r)/4$ be a small positive parameter to be chosen later, and
cover the circle $|z| = r$ with $J\asymp 1/\varepsilon$ discs $\mathcal{D}_j$ of radius $\varepsilon$ and centres $z_j$, where $z_j\in \{z: |z|=r\} $ for all $j\leq J$.

Let $\mathcal{S}_V(T)$ denote the set of those $t \in [T, 2T]$
such that
$$
\max_{|z| \leq (r+1/4)/2} |\zeta'(\tfrac 34 + it + z)| \leq  e^V
$$
where $V\leq \log\log T$ is a large parameter to be chosen later, and let $L := \max_{|z| \leq (r+1/4)/2} |f'(z)|$.
Then for $t \in \mathcal{S}_V(T)$, and for all $z\in \mathcal{D}_j$ we have
\begin{equation}\label{SmallDistance}
\begin{aligned}
& \Big|\zeta(\tfrac 34 +it+z)-f(z) - \big(\zeta(\tfrac 34 + it+ z_j)-f(z_j)\big)\Big|
= \left|\int_{z_j}^{z} \zeta'(\tfrac 34 + it+ s)-f'(s) ds\right| \\
& \leq |z - z_j| \cdot \left(\max_{|z| \leq (r+1/4)/2} |\zeta'(\tfrac 34 + it+ z)|+L\right)
\leq \varepsilon (e^V+L) \leq C\varepsilon e^V,
\end{aligned}
\end{equation}
for some large absolute constant $C$, depending at most on $L$. Define
$$ 
\theta(t):= \max_{|z| \leq r} |\zeta(\tfrac 34 + it + z) - f(z)|- \max_{j\leq J}|\zeta(\tfrac 34 + it + z_j) -
f(z_j)|.$$ Then, it follows from \eqref{ComparisonSupMax} and \eqref{SmallDistance} that for all $t\in \mathcal{S}_V(T)$ we have
\begin{equation}\label{BoundErrorSupMax}
0\leq \theta(t)\leq C\varepsilon e^V.
\end{equation}
Therefore, using this estimate together with Proposition \ref{ControlDerivative} and the fact that $\omega$ is bounded,  we deduce that \eqref{toestimate} equals 
\begin{equation}\label{SplitSmall}
\begin{aligned}
&\frac{1}{T}\int_{t\in \mathcal{S}_V(T)}\omega\left(\max_{j\leq J}|\zeta(\tfrac 34 + it + z_j)-
f(z_j)|+\theta(t)\right)dt+ O\left(e^{-V^2}\right)\\
&=\frac{1}{T}\int_{t\in \mathcal{S}_V(T)}\omega\left(\max_{ j\leq J}|\zeta(\tfrac 34 + it + z_j)-
f(z_j)|\right)dt+ O\left(|\mathcal{E}_2|+ e^{-V^2}\right)\\
&=\frac{1}{T}\int_{T}^{2T}\omega\left(\max_{j\leq J}|\zeta(\tfrac 34 + it + z_j)-
f(z_j)|\right)dt+ O\left(|\mathcal{E}_2|+ e^{-V^2}\right),\\
\end{aligned}
\end{equation}
where 
$$\mathcal{E}_2= \frac{1}{T}\int_{t\in \mathcal{S}_V(T)} \int_0^{\theta(t)} \omega'\left(\max_{j\leq J}|\zeta(\tfrac 34 + it + z_j)-
f(z_j)|+x\right) dx \cdot dt
\ll \varepsilon e^V, $$
using the fact that $\omega'$ is bounded on $\mathbb{R}$ together with \eqref{BoundErrorSupMax}. 

Furthermore, observe that 
\begin{equation}\label{SmoothProb}
\begin{aligned}
&\frac{1}{T}\int_{T}^{2T}\omega\left(\max_{j\leq J}|\zeta(\tfrac 34 + it + z_j)-
f(z_j)|\right)dt\\
=&
-\frac{1}{T}\int_{T}^{2T}\int_{\max_{j\leq J}|\zeta(\frac 34 + it + z_j)-
f(z_j)|}^{\infty}\omega'(u) du \cdot dt\\
=&
-\int_{0}^{\infty}\omega'(u) \cdot \mathbb P_T\left(\max_{j\leq J}|\zeta(\tfrac 34 + it + z_j)-
f(z_j)|\leq u\right) du.
\end{aligned}
\end{equation}
Since $\omega$ has a compact support, then $\omega'(u)=0$ if $u>A$ for some positive constant $A$. Furthermore, it follows from Proposition \ref{DistributionMax} that for all $0\leq u\leq A$ we have 
\begin{align*}
\mathbb P_T\left(\max_{j\leq J}|\zeta(\tfrac 34 + it + z_j)-
f(z_j)|\leq u\right)&= \mathbb P\left(\max_{j\leq J}|\zeta(\tfrac 34 + z_j, X)-
f(z_j)|\leq u\right)\\
& +O\left(\frac{(J\log\log T)^{6/5}}{(\log T)^{(3/4-r)/5}}\right).
\end{align*}
Inserting this estimate in \eqref{SmoothProb} gives that
\begin{equation}\label{ApproximationMAX}
\begin{aligned} 
&\frac{1}{T}\int_{T}^{2T}\omega\left(\max_{j\leq J}|\zeta(\tfrac 34 + it + z_j)-
f(z_j)|\right)dt\\
&= -\int_{0}^{\infty}\omega'(u) \cdot \mathbb P\left(\max_{j\leq J}|\zeta(\tfrac 34 + z_j, X)-
f(z_j)|\leq u\right) du+ O\left(\frac{(J\log\log T)^{6/5}}{(\log T)^{(3/4-r)/5}}\right)\\
&= \ex\left(\omega\left(\max_{j\leq J}|\zeta(\tfrac 34 + z_j, X)-
f(z_j)|\right) \right)+ O\left(\frac{(J\log\log T)^{6/5}}{(\log T)^{(3/4-r)/5}}\right).\\
\end{aligned}
\end{equation}

To finish the proof, we shall appeal to the same argument used to establish \eqref{SplitSmall}, in order to compare the (weighted) distributions of $\max_{j\leq J}|\zeta(\tfrac 34 + z_j, X)-f(z_j)|$ and $\max_{|z|\leq r}|\zeta(\tfrac 34 + z, X)-
f(z)|$. Let $\mathcal{S}_V(X)$ denote the event corresponding to 
$$ 
\max_{|z| \leq (r+1/4)/2} |\zeta'(\tfrac 34 + z, X)| \leq e^V,
$$  
and let $\mathcal{S}_V^c(X)$ be its complement. Then, it follows from Proposition \ref{DerRandom} that $\mathbb{P}\big(\mathcal{S}_V^c(X)\big)\ll \exp(-V^2).$ Moreover, similarly to \eqref{SmallDistance} one can see that for all outcomes in $\mathcal{S}_V(X)$ we have, for all $z\in \mathcal{D}_j$ 
$$\Big|\zeta(\tfrac 34+z, X)-f(z) - \big(\zeta(\tfrac 34 + z_j, X)-f(z_j)\big)\Big|
= \left|\int_{z_j}^{z} \zeta'(\tfrac 34 + s, X)-f'(s) ds\right| \ll \varepsilon e^V.
$$
Thus, since the maximum of $|\zeta(\tfrac 34 + z, X) - f(z)|$ for $|z|\leq r$ occurs (almost surely) on the boundary $|z|=r$, then following the argument leading to \eqref{SplitSmall}, we conclude that
\begin{align*}
&\ex\left(\omega\left(\max_{|z|\leq r}|\zeta(\tfrac 34 + z, X)-
f(z)|\right)\right)\\
&=\ex \left( \mathbf 1_{\mathcal S_V(X)} \,  \omega\left(\max_{|z|\leq r}|\zeta(\tfrac 34 + z, X)-
f(z)| \right) \right)+ O\left(e^{-V^2}\right)\\
&=\ex \left( \mathbf 1_{\mathcal S_V(X)} \, \omega\left(\max_{j\leq J}|\zeta(\tfrac 34 + z_j, X)-
f(z)| \right) \right)+ O\left(\varepsilon e^V+ e^{-V^2}\right)\\
&= \ex\left(\omega\left(\max_{j\leq J}|\zeta(\tfrac 34 + z_j, X)-
f(z)|\right)\right)+ O\left(\varepsilon e^V+ e^{-V^2}\right).\\
\end{align*}
Finally, combining this estimate with \eqref{SplitSmall} and \eqref{ApproximationMAX}, and noting that $J\asymp 1/\varepsilon $ we deduce that
\begin{align*}
\frac{1}{T}\int_{T}^{2T}\omega\left(\max_{|z| \leq r} |\zeta(\tfrac 34 + it + z) - f(z)|  \right)dt&=\ex\left(\omega\left(\max_{|z|\leq r}|\zeta(\tfrac 34 + z, X)-
f(z)|\right)\right)\\
&+O\left(\varepsilon e^V+ e^{-V^2}+O\left(\frac{(\log\log T)^{6/5}}{\varepsilon^{6/5}(\log T)^{(3/4-r)/5}}\right) \right).
\end{align*}
Choosing $\varepsilon=(\log T)^{-(3/4-r)/11}$ and $V=2\sqrt{\log\log T}$ completes the proof.
\end{proof}
\section{Controlling the derivatives of the zeta function and the random model: Proof of Propositions \ref{ControlDerivative} and \ref{DerRandom}} \label{sec:derivative}

By Cauchy's theorem we have
$$
|\zeta'(\tfrac 34 + it + z)| \leq \frac{1}{\delta}\max_{|s-z|=\delta} |\zeta(\tfrac 34 + it + s)|,
$$
and hence we get
\begin{equation}\label{Cauchy}
\max_{|z|\leq r} |\zeta'(\tfrac 34 + it + z)| \leq
\frac{1}{\delta} \max_{|s| \leq r + \delta} |\zeta(\tfrac 34 + it + s)|.
\end{equation}
Therefore, it follows that
\begin{equation} \label{derivative}
\begin{split}
\mathbb{P}_T \left( \max_{|z| \leq r} |\zeta'(\tfrac 34 + it + z)| > 
e^V  \right) &\leq 
\mathbb{P}_T \left( \max_{|s| \leq r + \delta} |\zeta(\tfrac 34 + it + s)| > 
\delta e^V \right) \\
&=\mathbb{P}_T \left(  \max_{|s| \leq r + \delta} \log |\zeta(\tfrac 34 + it + s)| > 
V+\log \delta \right).
\end{split}
\end{equation}
To bound the RHS we estimate large moments of $\log \zeta(\tfrac34+it+s)$. This is accomplished by approximating
 $\log \zeta(\tfrac 34 + it + s)$ by a short Dirichlet polynomial, uniformly for all $s$ in the disc $\{|s|\leq r+\delta\}$. Using zero density estimates and large sieve inequalities, we can show that such an approximation holds for all $t\in [T, 2T]$, except for an exceptional set of $t$'s with  very small measure. We prove
\begin{lemma}\label{UnifShortDirichlet}
Let $0<r<1/4$ be fixed, and $\delta= (1/4-r)/4$. Let $y\leq \log T$ be a real number. There exists a set $\mathcal{I}(T)\subset [T, 2T]$ with $\text{meas}(\mathcal{I}(T))\ll T^{1-\delta} y(\log T)^5$, such that for all $t\in [T, 2T]\setminus\mathcal{I}(T)$ and all $|s|\leq r+\delta$ we have 
$$
\log \zeta(\tfrac 34 +it+s)=\sum_{ n \le y} \frac{\Lambda(n)}{n^{\frac34+it+s} \log n}+O\left(
\frac{(\log y)^2 \log T}{y^{(1/4-r)/2}}\right).
$$
\end{lemma}
To prove this result, we need the following lemma from Granville and Soundararajan \cite{GrSo}.
\begin{lemma}[Lemma 1 of  \cite{GrSo}]\label{ApproxShortEuler}
Let $y\geq 2$ and $|t|\geq y+3$ be real numbers. Let $1/2\leq \sigma_0< 1$ and suppose that the rectangle
$\{z: \sigma_0<\textup{Re}(z)\leq 1, |\textup{Im}(z)-t|\leq y+2\}$ is free of zeros of $\zeta(z)$.  Then for any $\sigma$ with
$\sigma_0+2/\log y<\sigma\leq 1$ we have 
$$
\log \zeta(\sigma+it)=\sum_{n\leq y}\frac{\Lambda(n)}{n^{\sigma+it}\log n} +O\left(\log |t| \frac{(\log y)^2}{y^{\sigma-\sigma_0}}\right).
$$
\end{lemma}

\begin{proof}[Proof of Lemma \ref{UnifShortDirichlet}]
Let $\sigma_0= 1/2+\delta$. For $j=1, 2$ let $\mathcal{T}_j$ be the set of those $t \in [T, 2T]$ for which
the rectangle 
$$
\{ z : \sigma_0 < \textup{Re}(z) \leq 1 , |\textup{Im}(z) - t| < y+ 1+ j \}
$$
is free of zeros of $\zeta(z)$. Then, note that $\mathcal{T}_2\subseteq \mathcal{T}_1$, and for all  $t\in\mathcal{T}_2 $, we have $t+\im(s)\in \mathcal{T}_1$ for all $|s|\leq r+\delta $. Hence, by Lemma \ref{ApproxShortEuler}
we have 
$$
\log \zeta(\tfrac 34 +it+s) = \sum_{n \leq y} \frac{\Lambda(n)}{n^{3/4+ it+s}
\log n} + O \left( \frac{(\log y)^2 \log T}{y^{(1/4-r)/2}}\right),
$$
for all $t\in \mathcal{T}_2$ and all $|s|\leq r+\delta$. Let $N(\sigma, T)$ be the number of zeros of $\zeta(s)$ in the rectangle $\sigma<\re(s)\leq 1$ and $|\im(s)|\leq T$.  By the classical zero density estimate  $N(\sigma, T)\ll T^{3/2-\sigma}(\log T)^5$ (see for example Theorem 9.19 A ot Titchmarsh \cite{Ti})
we deduce that the measure of the complement of $\mathcal{T}_2$ in $[T, 2T]$
is $\ll T^{1- \delta}y(\log T)^5$. 
\end{proof}

We also require a minor variant of Lemma 3.3 of \cite{LLR}, whose proof we will omit.

\begin{lemma}  \label{lem:momentbd} 
Fix $1/2<\sigma<1$, and let $s$ be a complex number such that $\re(s)=\sigma,$ and $|\im(s)|\le1$. Then, for any positive integer
$k \le \log T/(3 \log y)$ we have
\[
\frac{1}{T} \int_{T}^{2T} \left| \sum_{n \le y}  \frac{\Lambda(n)}{n^{s+it} \log n} \right|^{2k} \, dt \ll \left(\frac{c_8 k^{1-{\sigma}}}{(\log k)^{\sigma}} \right)^{2k}
\]
and
\[
\mathbb E \left(\left| \log \zeta(s,X)  \right|^{2k}\right)
\ll \left(\frac{c_8 k^{1-{\sigma}}}{(\log k)^{\sigma}} \right)^{2k}
\]
for some positive constant $c_8$ that depends at most on $\sigma$. 
\end{lemma}

\begin{proof}[Proof of Proposition \ref{ControlDerivative}]
Let $\delta=e^{-V/2}$. Taking $y=(\log T)^{5(1/4-r)^{-1}}$ in Lemma \ref{UnifShortDirichlet}  gives for all $t\in [T, 2T]$ except for a set with measure $\ll T^{1- (1/4-r)/5}$ that
\begin{equation}\label{ZeroDensityApp} 
\log \zeta(\tfrac 34 +it+s) = \sum_{n \leq y} \frac{\Lambda(n)}{n^{3/4+ it+s}
\log n}+ O \left(\frac{1}{\log T}\right),
\end{equation}
for all $|s|\leq r+\delta$. 
Furthermore, it follows from Cauchy's integral formula that 
\[
\left( \sum_{n \le y} \frac{\Lambda(n)}{n^{3/4+it+s} \log n} \right)^{2k}=
\frac{1}{2\pi i} \int_{|z|=r+2\delta} \left( \sum_{n \le y} \frac{\Lambda(n)}{n^{3/4+it+z} \log n} \right)^{2k} \frac{dz}{z-s}.
\]
Applying Lemma \ref{lem:momentbd}  we get that
\begin{equation} \label{eq:cauchy}
\begin{split}
\frac{1}{T}
\int_T^{2T}
\left(\max_{|s| \le r+\delta} \left|  \sum_{n \le y} \frac{\Lambda(n)}{n^{3/4+it+s} \log n} \right| \right)^{2k} \, dt
\ll & \frac{1}{\delta} \int_{|z|=r+2\delta} \frac{1}{T} \int_{T}^{2T} \left| \sum_{n \le y}  \frac{\Lambda(n)}{n^{3/4+it+z} \log n} \right|^{2k} \, dt |dz| \\
\ll  & e^{V/2}  \left(c_8(r) \frac{k^{1-\sigma'(r)}}{(\log k)^{\sigma'(r)}} \right)^{2k}
\end{split}
\end{equation}
where $\sigma'(r)=\tfrac34-r-2\delta$, and $k \le   c_9 \log T/\log \log T$, for some sufficiently small constant $c_9>0$.
We now choose $k=\lfloor c_6(r) V^{\frac{1}{(1-\sigma(r))}} (\log V)^{\frac{\sigma(r)}{1-\sigma(r)}} \rfloor$ (so that $k^{\sigma'(r)} \asymp k^{\sigma(r)}$) where $c_6(r)$ is a sufficiently small absolute constant.
Using \eqref{derivative} and \eqref{ZeroDensityApp} along with Chebyshev's inequality and the above estimate we conclude that
there exists $c_7(r)>0$ such that
\[
\begin{split}
\mathbb P_T\bigg( \max_{|z| \le r} |\zeta'(\tfrac34+it+z)| > e^V \bigg)
\ll & \mathbb P_T\bigg( \max_{|s| \le r+\delta}  \left| \sum_{n \le y}  \frac{\Lambda(n)}{n^{3/4+it+s} \log n} \right| > \frac{V}{4} \bigg)+T^{1- (1/4-r)/5} \\
\ll & e^{V/2} \bigg(\frac{4}{V} \cdot c_8(r)\frac{ k^{1-\sigma'(r)}}{ (\log k)^{\sigma'(r)}}\bigg)^{2k}+T^{1- (1/4-r)/5}\\
\ll &  \exp\left(-c_6 V^{\frac{1}{1-\sigma(r)}} (\log V)^{\frac{\sigma(r)}{1-\sigma(r)}}\right)
\end{split}
\]
for $V \le c_7 (\log T)^{1-\sigma(r)}/\log \log T$.

\end{proof}

We now prove Proposition \ref{DerRandom} along the same lines. The proof is in fact easier than in the zeta function case, since we can compute the moments of $\log \zeta(s, X)$, for any $s$ with $\re(s)>1/2$.
\begin{proof}[Proof of Proposition \ref{DerRandom}] Let $\delta= e^{-V/2}$. Since $\zeta(\tfrac34+s, X)$ is almost surely analytic in $|s|\leq r+2\delta$, then by Cauchy's estimate we have almost surely that
$$
\max_{|z|\leq r} |\zeta'(\tfrac 34 + z, X)| \leq
\frac{1}{\delta} \max_{|s|\leq r + \delta} |\zeta(\tfrac 34 + s, X)|.
$$
Therefore, we obtain
\begin{equation} \label{derivativeRand}
\begin{split}
\mathbb{P} \left( \max_{|z|\leq r} |\zeta'(\tfrac 34 + z, X)| > 
e^V  \right)\leq &
\mathbb{P} \left( \max_{|s|\leq r+\delta} |\zeta(\tfrac 34 + s, X)| > 
 \delta e^V \right) \\
\le & \mathbb{P} \left( \max_{|s|\leq r+\delta} |\log \zeta(\tfrac 34 + s, X)| > 
 \frac{V}{2} \right). 
 \end{split}
\end{equation}

Let $k$ be a positive integer. By 
\eqref{MainRandom}
$\log \zeta(\tfrac34+s,X)$ converges almost surely
to a holomorphic function in $|s|\leq r+2\delta$.
Using Cauchy's integral formula as in \eqref{eq:cauchy}, we obtain almost surely that
$$
\left(\max_{|s|\leq r+\delta} |\log \zeta(\tfrac 34 + s, X)|\right)^{2k}
\ll \frac{1}{\delta} \int_{|z|=r+2\delta} \left| \log \zeta(\tfrac 34 + z, X)\right|^{2k}\cdot |dz|.
$$
Hence,  applying Lemma 
\ref{lem:momentbd} we get
\begin{equation}\label{BoundProbSubHarm}
\begin{aligned}
\mathbb{P} \left(\max_{|s|\leq r+\delta} |\log \zeta(\tfrac 34 + s, X)| > 
  V/2 \right)
 &\leq \left( \frac{2}{V}\right)^{2k} \cdot \ex\left(\left(\max_{|s|\leq r+\delta} |\zeta(\tfrac 34 + s, X)| \right)^{2k} \right)\\
 & \ll \left( \frac{2}{V}\right)^{2k} e^{V/2} \int_{|z|=r+2\delta} \ex\left(|\log \zeta(\tfrac 34 + z, X)|^{2k}\right) \cdot |dz| \\
 & \ll e^{V/2} \left(\frac{2 c_8(r) k^{1-\sigma'(r)}}{V (\log k)^{\sigma'(r)}}\right)^{2k},
 \end{aligned}
 \end{equation}
 where $\sigma'(r)=\tfrac34-r-2\delta$. Let $\sigma(r)=\tfrac34-r$ and
 take $k=\lfloor c_6 V^{\frac{1}{1-\sigma(r)}} (\log V)^{\frac{\sigma(r)}{1-\sigma(r)}} \rfloor$, where $c_6$ is sufficiently small (note that $k^{\sigma'(r)} \asymp k^{\sigma(r)}$), then apply \eqref{BoundProbSubHarm} to complete the proof.

\end{proof}

\section{Moments of joint shifts of $\log \zeta(s)$: Proof of Theorem \ref{MomentsShifts} }\label{sec:moments}

The proof of Theorem \ref{MomentsShifts} splits into two parts.
In the first part we derive an approximation to
$$
\prod_{j = 1}^{k} \log \zeta(s_j + it)
$$
by a short Dirichlet polynomial.
In the second part we compute the resulting mean-values and obtain
Theorem \ref{MomentsShifts}. 
\subsection{Approximating $\prod_{j = 1}^{k} \log \zeta(s_j + it)$ by short Dirichlet polynomials}

Fix $1/2<\sigma_0<1$, and let $\delta:=\sigma_0-1/2$. Let $k\leq \log T$ be a positive integer and $s_1, s_2, \dots, s_k$ be complex numbers (not necessarily distinct) 
in the rectangle $\{ z: \sigma_0\leq \re(z)<1, \text{ and }|\im(z)|\leq T^{\delta/4}\}$.  We let ${\mathbf s}=(s_1,\dots, s_k)$, and define
$$
F_{{\mathbf s}}(n)= \sum_{\substack{n_1,n_2,\dots,n_k\geq 2\\ n_1n_2\cdots n_k=n}}\prod_{\ell=1}^k \frac{\Lambda(n_{\ell})}{n_{\ell}^{s_{\ell}}\log(n_{\ell})}.
$$
Then for all complex numbers $z$ with $\re(z)>1-\sigma_0$ we have 
$$ 
\prod_{\ell=1}^k\log\zeta(s_{\ell}+z)= \sum_{n=1}^{\infty}\frac{F_{{\mathbf s}}(n)}{n^z}.$$
The main result of this subsection is the following proposition. 
\begin{proposition}\label{Approx}
Let $T$ be large, $s_1, \dots, s_k$ be as above, and $\mathcal{E}(T)$ be as in Lemma \ref{ExceptionalSet} below. Then, there exist positive constants $a(\sigma_0), b(\sigma_0)$ such that if $k\leq a(\sigma_0)(\log T)/\log\log T$ and $t\in [T, 2T]\setminus \mathcal{E}(T)$ then
 $$\prod_{j=1}^k\log\zeta(s_j+it) = \sum_{n\leq T^{\delta/8}} \frac{F_{\mathbf s}(n)}{n^{it}}+ O\left(T^{-b(\sigma_0)}\right).$$
\end{proposition}
This depends on a sequence of fairly standard lemmas which we now describe.

\begin{lemma}\label{DirichletC} With the same notation as above, we have 
$$|F_{{\mathbf s}}(n)|\leq \frac{(2\log n)^k}{n^{\sigma_0}}.$$
\end{lemma}
\begin{proof}
We have 
$$|F_{{\mathbf s}}(n)|\leq \frac{1}{n^{\sigma_0}(\log 2)^k}\sum_{\substack{n_1,n_2,\dots,n_k\geq 2\\ n_1n_2\cdots n_k=n}}\prod_{\ell=1}^k \Lambda(n_{\ell})\leq \frac{2^k}{n^{\sigma_0}}\left(\sum_{m|n}\Lambda (m)\right)^k\leq \frac{(2\log n)^k}{n^{\sigma_0}}.
$$
\end{proof}

\begin{lemma}\label{BoundLogZ}
Let $y\geq 2$ and $|t|\geq y+3$ be real numbers. Suppose that the rectangle $\{z: \sigma_0-\delta/2<\re(z)\leq 1, |\im(z)-t|\leq y+2\}$ is free of zeros of $\zeta(z)$. Then, for all complex numbers $s$ such that $\re(s)\geq \sigma_0-\delta/4$ and $|\im(s)|\leq y$ we have 
$$\log\zeta(s+it)\ll_{\sigma_0} \log|t|.$$
\end{lemma}

\begin{proof}
This follows from Theorem 9.6 B of Titchmarsh. 
\end{proof}

\begin{lemma}\label{ExceptionalSet}
Let  $s_1, \dots, s_k$ be as above. Then, there exists a set $\mathcal{E}(T)\subset [T, 2T]$ with measure $\text{meas}(\mathcal{E}(T))\ll T^{1-\delta/8}$, and such that for all $t\in [T, 2T]\setminus \mathcal{E}(T)$ we have $\zeta(s_j+it+z)\neq 0$ for every $1\leq j\leq k$ and every $z$ in the rectangle $\{z: -\delta/2<\re(z)\leq 1, |\im(z)|\leq 3T^{\delta/4}\}$.
\end{lemma}
\begin{proof}
For every $1\leq j\leq k$, let $\mathcal{E}_j(T)$ be the set of $t\in [T, 2T]$ such that the rectangle $\{z: -\delta/2<\re(z)\leq 1, |\im(z)|\leq 3T^{\delta/4}\}$ has a zero of $\zeta(s_j+it+z)$. Then, by the classical zero density estimate $N(\sigma, T)\ll T^{3/2-\sigma}(\log T)^5$, we deduce that 
$$ \textup{meas}(\mathcal{E}_j(T))\ll T^{\delta/4} T^{3/2-\sigma_0+\delta/2}(\log T)^5 < T^{1-\delta/4}(\log T)^5.$$
We take $\mathcal{E}(T)=\cup_{j=1}^k \mathcal{E}_j(T)$. Then $\mathcal{E}(T)$ satisfies the assumptions of the lemma, since $\textup{meas}(\mathcal{E}(T))\ll T^{1-\delta/4}(\log T)^6\ll T^{1-\delta/8}$.
\end{proof}

We are now ready to prove
Proposition \ref{Approx}. 

\begin{proof}[Proof of Proposition \ref{Approx}]
Let $x=\lfloor T^{\delta/8}\rfloor +1/2$. Let $c=1-\sigma_0+ 1/\log T$, and $Y=T^{\delta/4}$. Then by Perron's formula, we have for $t\in [T, 2T]\setminus \mathcal{E}(T)$
$$ \frac{1}{2\pi i}\int_{c-iY}^{c+iY}  \left(\prod_{j=1}^k\log\zeta(s_j+it+z)\right)\frac{x^z}{z}dz= 
\sum_{n\leq x} \frac{F_{{\mathbf s}}(n)}{n^{it}}+ O\left(\frac{x^c}{Y}\sum_{n=1}^{\infty} \frac{|F_{{\mathbf s}}(n)|}{n^{c}|\log(x/n)|}\right).$$
To bound the error term of this last estimate, we split the sum into three parts: $n\leq x/2$, $x/2<n<2x$ and $n\geq 2x$. The terms in the first and third parts satisfy $|\log(x/n)|\geq \log 2$, and hence their contribution is 
$$\ll \frac{x^{1-\sigma_0}}{Y} \sum_{n=1}^{\infty}\frac{|F_{{\mathbf s}}(n)|}{n^{c}}\leq \frac{x^{1-\sigma_0}}{Y} \left(\sum_{n=1}^{\infty}\frac{\Lambda(n)}{n^{\sigma_0+c}\log n}\right)^k\leq \frac{x^{1-\sigma_0}(2\log T)^k}{Y}\ll T^{-b(\sigma_0)},$$
or some positive constant $b(\sigma_0)$, if $a(\sigma_0)$ is sufficiently small.
To handle the contribution of the terms $x/2<n<2x$, we put $r=x-n$, and use that $|\log(x/n)|\gg |r|/x$. Then by Lemma \ref{DirichletC} we deduce that the contribution of these terms is 
$$\ll \frac{x^{1-\sigma_0}(3\log x)^k}{Y}\sum_{r\leq x}\frac{1}{r}\ll \frac{x^{1-\sigma_0}(3\log x)^{k+1}}{Y}\ll T^{-b(\sigma_0)}.$$
We now move the contour to the line $\re(s)=-\delta/4$. By Lemma \ref{ExceptionalSet}, we do not encounter any zeros of $\zeta(s_j+it+z)$ since $t\in [T, 2T]\setminus \mathcal{E}(T)$.  We pick up a simple pole at $z=0$ which leaves a residue $\prod_{j=1}^k\log\zeta(s_j+it)$. 
Also  Lemma \ref{BoundLogZ} implies that for any $z$ on our contour we have
$$|\log\zeta(s_j+it+z)|\leq c(\sigma_0) \log T,$$
for all $j$ where $c(\sigma_0)$ is a positive constant. Therefore, we deduce that 
$$ \frac{1}{2\pi i}\int_{c-iY}^{c+iY} \left(\prod_{j=1}^k\log\zeta(s_j+it+z)\right)\frac{x^z}{z}dz= \prod_{j=1}^k\log\zeta(s_j+it) + E_1,$$
where 
\begin{align*}
 E_1&=\frac{1}{2\pi i} \left(\int_{c-iY}^{-\delta/4-iY}+ \int_{-\delta/4-iY}^{-\delta/4+iY}+ \int_{-\delta/4+iY}^{c+iY}\right) \left(\prod_{j=1}^k\log\zeta(s_j+it+z)\right)\frac{x^z}{z}dz\\
&\ll \frac{x^{1-\sigma_0}(c(\sigma_0)\log T)^k}{Y}+ x^{-\delta/4}(c(\sigma_0)\log T)^k\log Y\ll T^{-b(\sigma_0)},
\end{align*}
as desired.

\end{proof}

\subsection{An Asymptotic formula for the moment of products of shifts of $\log\zeta(s)$}

\begin{proof}[Proof of Theorem \ref{MomentsShifts}]
Let $\mathcal{E}_1(T)$ and $\mathcal{E}_2(T)$ be the corresponding exceptional sets for ${\mathbf s}$ and ${\mathbf r}$ respectively as in Lemma \ref{ExceptionalSet}, and let $\mathcal{E}(T)= \mathcal{E}_1(T)\cup \mathcal{E}_2(T)$. First, note that if $t\in [T,2T]\setminus \mathcal{E}(T)$ then by Proposition \ref{Approx} and Lemma \ref{BoundLogZ} we have 
$$ \left|\sum_{n\leq x} \frac{F_{\mathbf s}(n)}{n^{it}}\right|\ll (c(\sigma_0)\log T)^{k}, \text{ and } \left|\sum_{m\leq x} \frac{F_{\mathbf r}(m)}{m^{-it}}\right|\ll (c(\sigma_0))\log T)^{\ell},$$
for some positive constant $c(\sigma_0)$. Let $x=T^{(\sigma_0-1/2)/8}$. 
Then, it follows from Proposition \ref{Approx} that
\begin{equation}\label{MomentsProduct}
\begin{aligned}
&\frac{1}{T} \int_{[T,2T]\setminus \mathcal{E}(T)}\left(\prod_{j=1}^k\log\zeta(s_j+it)\right)\left(\prod_{j=1}^{\ell}\log\zeta(r_j-it)\right)dt\\
&= \frac{1}{T} \int_{[T,2T]\setminus \mathcal{E}(T)} \left(\sum_{n\leq x} \frac{F_{\mathbf s}(n)}{n^{it}}\right)\left(\sum_{m\leq x} F_{\mathbf r}(m)m^{it}dt\right)  dt + O\left(T^{-b(\sigma_0)}(\log T)^{\max(k, \ell)}\right)\\
&= \frac{1}{T} \int_T^{2T} \left(\sum_{n\leq x} \frac{F_{\mathbf s}(n)}{n^{it}}\right)\left(\sum_{m\leq x} F_{\mathbf r}(m)m^{it}\right)dt + O\left(T^{-b(\sigma_0)/2}\right).
\end{aligned}
\end{equation}

Furthermore, we have 

\begin{equation} \label{eq:mvt}
\frac{1}{T} \int_T^{2T} \left(\sum_{n\leq x} \frac{F_{\mathbf s}(n)}{n^{it}}\right)\left(\sum_{m\leq x} F_{\mathbf r}(n)m^{it}\right)dt= \sum_{m,n\leq x} F_{\mathbf s}(n)F_{\mathbf r}(m)\frac{1}{T} \int_T^{2T}\left(\frac{m}{n}\right)^{it}dt.
\end{equation}
The contribution of the diagonal terms $n=m$ equals $\sum_{n\leq x} F_{\mathbf s}(n)F_{\mathbf r}(n)$. On the other hand, by Lemma \ref{DirichletC} the contribution of  the off-diagonal terms $n\neq m$ is 
\begin{equation} \label{eq:offdiag}
\ll \frac{1}{T}\sum_{\substack{m,n\leq x \\ m \neq n}} \frac{(2\log n)^k (2\log m)^{\ell}}{(mn)^{\sigma_0}}\frac{1}{|\log(m/n)|}\ll \frac{x^{3-2\sigma_0}(2\log x)^{k+\ell}}{T}\ll T^{-1/2},
\end{equation}
since $|\log(m/n)|\gg 1/x$.

Furthermore, it follows from \eqref{orthogonality} that 
\begin{equation} \label{eq:randmvt}
\ex\left(\prod_{j=1}^k\log\zeta(s_j,X)\right)\left(\prod_{j=1}^{\ell}\log\overline{\zeta(r_j, X)}\right)= \sum_{n=1}^{\infty} F_{\mathbf s}(n)F_{\mathbf{r}}(n)= \sum_{n\leq x} F_{\mathbf s}(n)F_{\mathbf{r}}(n)+E_2,
\end{equation}
where 
$$E_2\leq \sum_{n>x}\frac{(2\log n)^{k+\ell}}{n^{2\sigma_0}}.$$
Since the function $(\log t)^{\beta}/t^{\alpha}$ is decreasing for $t>\exp(\beta/\alpha)$, then with the choice $\alpha=(2\sigma_0-1)/2$ we obtain
$$ E_2\leq \frac{(2\log x)^{k+\ell}}{x^{\alpha}}\sum_{n>x}\frac{1}{n^{1+\alpha}}\ll \frac{(2\log x)^{k+\ell}}{x^{2\alpha}}\ll x^{-\alpha}.$$
Combining this with \eqref{eq:mvt}, \eqref{eq:offdiag}, and \eqref{eq:randmvt} completes the proof.

\end{proof}

\section{The characteristic function of joint shifts of $\log \zeta(s)$} \label{sec:characteristic}

\begin{proof}[Proof of Theorem \ref{characteristic}]

Let $\mathcal{E}(T)$ be as in Theorem \ref{MomentsShifts}. Let $N=[\log T/(C(\log\log T))]$ where $C$ is a suitably large constant.  Then,  
$\Phi_T(\mathbf{u}, \mathbf{v})$ equals
\begin{align}\label{Taylor}
& \nonumber \frac1T \int_{[T, 2T]\setminus \mathcal{E}(T)} \exp\left(i\left(\sum_{j=1}^J (u_j \re \log\zeta(s_j+it)+ v_j \im \log\zeta (s_j+it))\right)\right)dt +O\left(T^{-c_3}\right)\\
& =\sum_{n=0}^{2N-1} \frac{i^n}{n!} \cdot \frac1T \int_{[T, 2T]\setminus \mathcal{E}(T)}\left(\sum_{j=1}^J (u_j \re \log\zeta(s_j+it)+ v_j \im \log\zeta (s_j+it))\right)^ndt + E_3,
\end{align}
where 
$$E_3 \ll T^{-c_3}+ \frac{1}{(2N)!}\left(\frac{2c_1(\log T)^{\sigma}}{J}\right)^{2N}\frac1T \int_{[T, 2T]\setminus \mathcal{E}(T)}  \left(\sum_{j=1}^J|\log\zeta(s_j+it)|\right)^{2N}dt.$$
Now, by Theorem \ref{MomentsShifts} along with Lemma \ref{lem:momentbd}, we obtain that for all $1\leq j\leq J$
\begin{equation}\label{BoundM}
 \frac1T \int_{[T, 2T]\setminus \mathcal{E}(T)}  |\log\zeta(s_j+it)|^{2N}dt \ll 
\ex\left(|\log\zeta(s_j, X)|^{2N}\right)\leq \left(\frac{c_8(\sigma) N^{1-\sigma}}{(\log N)^{\sigma}}\right)^{2N},
\end{equation}
for some positive constant $c_8=c_8(\sigma).$ Furthermore, by Minkowski's inequality we have
$$ \frac1T \int_{[T, 2T]\setminus \mathcal{E}(T)}  \left(\sum_{j=1}^J|\log\zeta(s_j+it)|\right)^{2N}dt \leq \left(c_8 J \frac{N^{1-\sigma}}{(\log N)^{\sigma}}\right)^{2N}.$$
Therefore, we deduce that for some positive constant $c_{9}=c_{9}(\sigma)$, we have
$$ E_3\ll T^{-c_3} + \left(c_{9}\frac{(\log T)^{\sigma}}{(N\log N)^{\sigma}}\right)^{2N}\ll e^{-N}.$$

Next, we handle the main term of \eqref{Taylor}. Let $\tilde{u_j}=(u_j+iv_j)/2$ and $\tilde{v_j}=(u_j-iv_j)/2$. Then by Theorem \ref{MomentsShifts} we obtain
\begin{align*}
&\frac1T \int_{[T, 2T]\setminus \mathcal{E}(T)}\left(\sum_{j=1}^J (u_j \re \log\zeta(s_j+it)+ v_j \im \log\zeta(s_j+it))\right)^ndt\\
&\frac1T \int_{[T, 2T]\setminus \mathcal{E}(T)}\left(\sum_{j=1}^J (\tilde{u_j} \log\zeta(s_j+it)+ \tilde{v_j}  \log\zeta(s_j-it))\right)^ndt\\
&=\sum_{\substack{k_1,\dots, k_{2J}\geq 0\\ k_1+\cdots+k_{2J}=n}}{n\choose k_1, k_2, \dots, k_{2J}}\prod_{j=1}^J\tilde{u_j}^{k_j} \prod_{\ell=1}^J\tilde{v_{\ell}}^{k_{J+\ell}}\\
& \quad \quad \times \frac1T \int_{[T, 2T]\setminus \mathcal{E}(T)} \prod_{j=1}^J(\log\zeta(s_j+it))^{k_j}\prod_{\ell=1}^J  (\log\zeta(s_\ell-it))^{k_{J+\ell}}dt\\
&= \sum_{\substack{k_1,\dots, k_{2J}\geq 0\\ k_1+\cdots+k_{2J}=n}}{n\choose k_1, k_2, \dots, k_{2J}}\prod_{j=1}^J\tilde{u_j}^{k_j} \prod_{\ell=1}^J\tilde{v_{\ell}}^{k_{J+\ell}}\\
& \quad \quad \times \ex\left(\prod_{j=1}^J(\log\zeta(s_j, X))^{k_j}\prod_{\ell=1}^J  (\log\zeta(s_\ell, X))^{k_{J+\ell}}\right) +O\Big(T^{-c_5} \big(2c_1(\log T)^{\sigma}\big)^n\Big),\\
&= \ex\left(\left(\sum_{j=1}^J (u_j \re \log\zeta(s_j, X)+ v_j \im \log\zeta(s_j, X))\right)^n\right) +O\Big(T^{-c_5} \big(2c_1(\log T)^{\sigma}\big)^n\Big).
\end{align*}

Inserting this estimate in \eqref{Taylor}, we derive
\begin{align*}
\Phi_T(\mathbf{u}, \mathbf{v})&=\sum_{n=0}^{2N-1} \frac{i^n}{n!}\ex\left(\left(\sum_{j=1}^J (u_j \re \log\zeta(s_j, X)+ v_j \im \log\zeta(s_j, X))\right)^n\right) + O\Big(e^{-N}\Big)\\ &= \Phi_{\text{rand}}(\mathbf{u}, \mathbf{v}) +E_4.
\end{align*}
where 
$$ E_4\ll e^{-N} + \frac{1}{(2N)!}\left(\frac{2c_1(\log T)^{\sigma}}{J}\right)^{2N}\ex\left(\left(\sum_{j=1}^J | \log\zeta(s_j, X)|\right)^{2N}\right)\ll e^{-N}
$$
by \eqref{BoundM} and Minkowski's inequality. This completes the proof.
\end{proof}

\section{Discrepancy estimates for the distribution of shifts} \label{sec:distribution}

The deduction of Theorem \ref{discrep} from Theorem \ref{characteristic} uses Beurling-Selberg functions.
For $z\in \mathbb C$ let
\[
H(z) =\bigg( \frac{\sin \pi z}{\pi} \bigg)^2 \bigg( \sum_{n=-\infty}^{\infty} \frac{\tmop{sgn}(n)}{(z-n)^2}+\frac{2}{z}\bigg)
\qquad\mbox{and} \qquad K(z)=\Big(\frac{\sin \pi z}{\pi z}\Big)^2.
\]
Beurling proved that the function $B^+(x)=H(x)+K(x)$
majorizes $\tmop{sgn}(x)$ and its Fourier transform
has restricted support in $(-1,1)$. Similarly, the function $B^-(x)=H(x)-K(x)$ minorizes $\tmop{sgn}(x)$ and its Fourier
transform has the same property (see Vaaler \cite{Vaaler} Lemma 5).

Let $\Delta>0$ and $a,b$ be real numbers with $a<b$. Take $\mathcal I=[a,b]$
and
define
\[
F_{\mathcal I} (z)=\frac12 \Big(B^-(\Delta(z-a))+B^-(\Delta(b-z))\Big).
\]
The function $F_{\mathcal I}$  has the following remarkable properties. 
First, it follows from the inequality  $B^-(x) \le \tmop{sgn}(x) \le B^+(x)$ that
\begin{equation} \label{l1 bd}
0 \le \mathbf 1_{\mathcal I}(x)- F_{\mathcal I}(x)\le K(\Delta(x-a))+K(\Delta(b-x)).
\end{equation}
Additionally, one has
\begin{equation} \label{Fourier}
\widehat F_{\mathcal I}(\xi)=
\begin{cases}\widehat{ \mathbf 1}_{\mathcal I}(\xi)+O\Big(\frac{1}{\Delta} \Big) \mbox{ if } |\xi| < \Delta, \\
0 \mbox{ if } |\xi|\ge \Delta.
\end{cases}
\end{equation}
The first estimate above follows from \eqref{l1 bd} and
the second follows from the fact that the Fourier transform
of $B^-$ is supported in $(-1,1)$.
Before proving Theorem \ref{discrep} we first require the following lemmas.
\begin{lemma} \label{lem:functionbd}
For $x \in \mathbb R$ we have
$
|F_{\mathcal I}(x)| \le 1.
$
\end{lemma}
\begin{proof}
It suffices to prove the lemma for $\Delta=1$. Also, note that we only need to show that $F_{\mathcal I}(x) \ge -1$.
From the identity
\[
\sum_{ n=-\infty}^{\infty} \frac{1}{(n-z)^2}=
\left(\frac{\pi}{\sin \pi z}\right)^2
\]
it follows that for $y \ge 0$
\begin{equation}\label{eq:Hid}
H(y)=1-K(y)G(y)
\end{equation}
where
\[
G(y)=2y^2 \sum_{m=0}^{\infty} \frac{1}{(y+m)^2}-2y-1.
\]
In Lemma 5 of \cite{Vaaler}, Vaaler shows for $y \ge 0$ that
\begin{equation} \label{eq:Gbd}
0 \le G(y) \le 1.
\end{equation}
Also, note that for each $m \ge 1$, and $0<y \le 1$ one has $\frac{m}{(y+m)^3} \le \int_{m-1}^m \frac{t}{(y+t)^3} \, dt$ so that for $0<y \le 1$ 
\begin{equation} \label{eq:Gdec}
G'(y)=4y \sum_{m \ge 1} \frac{m}{(y+m)^3}-2 \le 4y \int_0^{\infty} \frac{t}{(y+t)^3} \, dt-2 = 0.
\end{equation}
 

First consider the case $a\le x \le b$. By \eqref{eq:Hid} we get that in this range
\[
F_{\mathcal I}(x)=\frac12 \left(2- K(x-a)(G(x-a)+1)-K(b-x)(G(b-x)+1) \right),
\]
which along with \eqref{eq:Gbd} implies $F_{\mathcal I}(x) \ge -1$ for $a \le x \le b$. 
Now consider the case $x<a$. Since $H$ is an odd function
\eqref{eq:Hid} and \eqref{eq:Gbd} imply 
\[
\begin{split}
F_{\mathcal I}(x)=& \frac12 \left(K(x-a)(G(a-x)-1)-K(b-x)(G(b-x)+1) \right) \\
\ge & \frac12\left( -K(x-a)-2K(x-b)\right),
\end{split}
\]
which is $\ge -1$ if $K(x-b) \le 1/2$. If $K(x-b) \ge 1/2$
we also have $K(x-a)>K(x-b)$ and $0<b-x < 1$.
By this and \eqref{eq:Gdec} we have in this range as well that
\[
F_{\mathcal I}(x) \ge \frac12 \left( K(x-b)(G(a-x)-G(b-x)-2)\right) \ge -1.
\]
Hence, $F_{\mathcal I}(x)\ge -1$ for $x<a$.
The remaining case when $x>b$ follows from a similar argument.
\end{proof}

\begin{lemma} \label{upper bd}
Fix $1/2<\sigma<1$, and let $s$ be a complex number such that $\tmop{Re}(s)=\sigma$ and $|\tmop{Im}(s)| \le T^{\frac14\cdot(\sigma-\frac12)}$. Then there exists a positive constant $c_1(\sigma)$ such that for $|u| \le c_1(\sigma)(\log T)^{\sigma}$ we have
\[
\Phi_T(u,0) \ll \exp\left( \frac{-u}{5 \log u} \right) \quad
\mbox{ and } \quad \Phi_T(0,u) \ll \exp\left( \frac{-u}{5 \log u} \right).
\]
\end{lemma}
\begin{proof}
By a straightforward modification of 
Lemma 6.3 of \cite{LLR} one has that
\[
\mathbb E \bigg( \exp\Big(i u  \tmop{Re} \log \zeta(s, X)\Big) \bigg) \ll \exp\bigg(-\frac{u}{5 \log u} \bigg),
\]
and 
\[
\mathbb E \bigg( \exp\Big(i u  \tmop{Im} \log \zeta(s, X)\Big) \bigg) \ll \exp\bigg(-\frac{u}{5 \log u} \bigg).
\]
Using the first bound and applying Theorem \ref{characteristic}  with $J=1$ establishes the first claim. The second claim follows similarly by using the second bound and Theorem \ref{characteristic}.
\end{proof}

\begin{proof}[Proof of Theorem \ref{discrep}]
First, we claim that it suffices to estimate
the discrepancy over $(\mathcal R_1, \ldots, \mathcal R_J)$ such 
that for each $j$ we have $\mathcal R_j \subset [-\sqrt{\log T}, \sqrt{\log T}] \times [-\sqrt{\log T}, \sqrt{\log T}]$.
To see this consider $( \widetilde{\mathcal R_1}, \ldots, \widetilde{\mathcal R_J})$ where 
$\widetilde{\mathcal R_j}=\mathcal R_j \cap [-\sqrt{\log T}, \sqrt{\log T}]  \times [-\sqrt{\log T}, \sqrt{\log T}] $.
It follows that
\begin{equation} \notag
\begin{split}
&\bigg|\mathbb P_T \bigg( \log \zeta(s_j+it) \in \mathcal R_j, \forall  j \le J \bigg)
-\mathbb P_T \bigg( \log \zeta(s_1+it) \in \widetilde{\mathcal R_1},\log \zeta(s_j+it) \in \mathcal R_j, 2 \le j \le J \bigg)
\bigg| \\
&\ll \mathbb P_T \bigg( |\log \zeta(s_1+it)| \ge \sqrt{\log T} \bigg) \ll \exp\Big(-\sqrt{\log T}\Big),
\end{split}
\end{equation}
where the last bound follows from Theorem 1.1 and Remark 1 of
of \cite{La}.
Repeating this argument gives
\[
\bigg|\mathbb P_T \bigg( \log \zeta(s_j+it) \in \widetilde{\mathcal R_j}, \forall  j \le J )
-\mathbb P_T \bigg( \log \zeta(s_j+it) \in \mathcal R_j,  \forall  j \le J \bigg)
\bigg| \ll J \exp\Big(-\sqrt{\log T}\Big).
\]
Similarly,
\[
\bigg|\mathbb P \bigg( \log \zeta(s_j,X) \in \widetilde{\mathcal R_j}, \forall  j \le J )
-\mathbb P \bigg( \log \zeta(s_j,X) \in \mathcal R_j,  \forall  j \le J \bigg)
\bigg| \ll J \exp\Big(-\sqrt{\log T}\Big).
\]
Hence, the error from restricting to $( \widetilde{\mathcal R_1}, \ldots, \widetilde{\mathcal R_J})$ is negligible and establishes
the claim.

Let $\Delta=c_1(\sigma) (\log T)^{\sigma}/J$ and $\mathcal R_j=[a_j,b_j]\times[c_j, d_j]$ for $j=1, \ldots, J$, with
$|b_j-a_j|,|d_j-c_j| \le 2\sqrt{\log T}$.
Also, write $\mathcal I_j=[a_j,b_j]$ and $ \mathcal J_j=[c_j,d_j]$.
By Fourier inversion, \eqref{Fourier}, and Theorem \ref{characteristic} we have that
\begin{equation} \label{long est}
\begin{split}
&\frac1T \int_T^{2T} \prod_{j=1}^J F_{\mathcal I_j} \Big( \tmop{Re} \log \zeta(s_j+it)\Big) 
F_{\mathcal J_j}\Big( \tmop{Im} \log \zeta(s_j+it)\Big) \, dt\\
&
=\int_{\mathbb R^{2J}} \bigg(\prod_{j=1}^J \widehat{F}_{\mathcal I_j} (u_j) 
\widehat{F}_{\mathcal J_j}( v_j)\bigg)  \Phi_T(\mathbf u, \mathbf v) \, d\mathbf u \, d\mathbf v \\
&
= \int\limits_{\substack{|u_j|,|v_j| \le \Delta \\ j=1,2, \ldots, J}} \bigg(\prod_{j=1}^J \widehat{F}_{\mathcal I_j} (u_j) 
\widehat{F}_{\mathcal J_j}( v_j)\bigg)  \Phi_{\tmop{rand}}(\mathbf u, \mathbf v) \, d\mathbf u \, d\mathbf v +O\left(\left(2\Delta\sqrt{\log T}\right)^{2J} \exp\Big(-
\frac{c_2 \log T}{\log \log T}\Big)\right)\\
& 
=\mathbb E \bigg( \prod_{j=1}^J F_{\mathcal I_j} \Big( \tmop{Re} \log \zeta(s_j,X)\Big) 
F_{\mathcal J_j}\Big( \tmop{Im} \log \zeta(s_j,X)\Big) \bigg)+O\left( \exp\left(-
\frac{c_2 \log T}{2\log \log T}\right)\right).
\end{split}
\end{equation}

Next note that $\widehat K(\xi)=\max(0,1-|\xi|)$. Applying Fourier inversion, Theorem \ref{characteristic} with $J=1$,
and Lemma \ref{upper bd} we have  that
\begin{equation} \notag
\begin{split}
\frac1T \int_T^{2T} K\Big( \Delta \cdot \Big(\tmop{Re} \log \zeta(s+it)-\alpha\Big)\Big) \, dt
=&\frac{1}{\Delta}\int_{-\Delta}^{\Delta}\Big(1-\frac{|\xi|}{\Delta}\Big) e^{-2\pi i \alpha \xi} \Phi_T(\xi,0) \, d\xi 
\ll  \frac{1}{\Delta},
\end{split}
\end{equation}
where $\alpha$ is an arbitrary real number and $s \in \mathbb C$ satisfies
 $\sigma \le \tmop{Re(s)} <1$ and $|\tmop{Im}(s)|< T^{\frac14(\sigma-\frac12)}$. 
By this and \eqref{l1 bd} we get that
\begin{equation} \label{K bd}
\frac1T \int_T^{2T}  F_{\mathcal I_1}\Big(\tmop{Re} \log \zeta(s_1+it)\Big) \, dt
=\frac1T \int_{T}^{2T} \mathbf 1_{\mathcal I_1}\Big(\tmop{Re} \log \zeta(s_1+it)\Big) dt+O(1/\Delta).
\end{equation}
Lemma \ref{lem:functionbd} implies that $|F_{\mathcal I_j}(x)|, |F_{\mathcal J_j}(x)| \le 1$ for  $j=1,\ldots, J$. Hence, by this and \eqref{K bd}
\begin{equation} \notag
\begin{split}
&\frac1T \int_T^{2T} \prod_{j=1}^J F_{\mathcal I_j} \Big( \tmop{Re} \log \zeta(s_j+it)\Big) 
F_{\mathcal J_j}\Big( \tmop{Im} \log \zeta(s_j+it)\Big) \, dt \\
&=\frac1T \int_T^{2T} \mathbf 1_{\mathcal I_1} \Big( \tmop{Re} \log \zeta(s_j+it)\Big) 
F_{\mathcal J_1}\Big( \tmop{Im} \log \zeta(s_j+it)\Big) \\
&\qquad \qquad \qquad \times \prod_{j=2}^J F_{\mathcal I_j} \Big( \tmop{Re} \log \zeta(s_j+it)\Big) 
F_{\mathcal J_j}\Big( \tmop{Im} \log \zeta(s_j+it)\Big) \, dt+O(1/\Delta).
\end{split}
\end{equation}
Iterating this argument and using an analog of \eqref{K bd} for $\tmop{Im } \log \zeta(s+it)$, which
is proved in the same way, gives
\begin{equation} \label{one}
\begin{split}
&\frac1T \int_T^{2T} \prod_{j=1}^J F_{\mathcal I_j} \Big( \tmop{Re} \log \zeta(s_j+it)\Big) 
F_{\mathcal J_j}\Big( \tmop{Im} \log \zeta(s_j+it)\Big) \, dt \\
&\qquad \qquad \qquad
=\mathbb P_T\Bigg(\log \zeta(s_j+it) \in \mathcal R_j, \forall j \le J\bigg) +O\left(\frac{J}{\Delta}\right).
\end{split}
\end{equation}
Similarly, it can be shown that
\begin{equation} \label{two}
\begin{aligned}
\mathbb E \bigg( \prod_{j=1}^J F_{\mathcal I_j} \Big( \tmop{Re} \log \zeta(s_j,X)\Big) 
F_{\mathcal J_j}\Big( \tmop{Im} \log \zeta(s_j,X)\Big) \bigg) 
&=\mathbb P\Bigg(\log \zeta(s_j,X) \in \mathcal R_j, \forall j \le J  \bigg) \\
&\ \ \ +O\left(\frac{J}{\Delta}\right).
\end{aligned}
\end{equation}
Using \eqref{one} and \eqref{two} in \eqref{long est} completes the proof.

\end{proof}

\end{document}